\documentclass{amsproc}

\usepackage{tikz}
\usepackage{mathdots}
\usetikzlibrary{automata,cd,shapes}
\usepackage{faktor} 

\usepackage{amssymb}
\usepackage{mathtools}

\makeatletter
\DeclareRobustCommand{\rvdots}{
  \vbox{
    \baselineskip4\p@\lineskiplimit\z@
    \kern-\p@
    \hbox{.}\hbox{.}\hbox{.}
  }}
\makeatother

\newtheorem{theorem}{Theorem}[section]
\newtheorem{lemma}[theorem]{Lemma}

\theoremstyle{definition}
\newtheorem{definition}[theorem]{Definition}
\newtheorem{example}[theorem]{Example}

\theoremstyle{remark}
\newtheorem{remark}[theorem]{Remark}
\newtheorem{fact}[theorem]{Fact}

\numberwithin{equation}{section}
\DeclareMathOperator{\Aut}{Aut}
\DeclareMathOperator{\aut}{aut}

\DeclareMathOperator{\IMG}{IMG}
\DeclareMathOperator{\Stab}{Stab}
\newcommand{\Esc}{\mathbf {I}}
\newcommand{\Ipart}{\mathbb {U}} 
\newcommand\C{\mathbb {C}} 
\newcommand\Z{\mathbb {Z}} 
\newcommand\N{\mathbb {N}} 
 
\newcommand\A{\mathbf {A}} 
\newcommand\T{\mathcal {T}} 
 
\newcommand\Sing{\mathbf {S}} 
\newcommand\Spider{\mathbb {S}} 
\newcommand\PP{\mathbf {P}} 
\newcommand\Id{\mathbf {1}}

\newcommand\Kgroup{\mathcal{K}} 
\newcommand\Kaut{\mathbf{K}} 
\newcommand{\qedcited}{\hfill \ensuremath{\triangle}}
\newcommand{\autBfs}{\Aut^{f.s.}_{\mathcal{B}}}
\begin{document}

\title{Iterated Monodromy Groups of Exponential Maps}

\author{Bernhard Reinke}
\address{Institut de Mathématiques (UMR CNRS7373) \\
Campus de Luminy \\
163 avenue de Luminy --- Case 907 \\
13288 Marseille 9 \\
France}
\curraddr{}
\email{}
\thanks{}



\date{}

\begin{abstract}
	This paper introduces iterated monodromy groups for transcendental
	functions and discusses them in the simplest setting, for
	post-singularly finite exponential functions. These groups are
	self-similar groups in a natural way, based on an explicit
	construction in terms of kneading sequences. We investigate the group
	theoretic properties of these groups, and show in particular that they
	are amenable, but they are not elementary subexponentially amenable.
\end{abstract}

\maketitle

\section{Introduction}
In the iteration theory of rational maps, iterated monodromy groups are
self-similar groups associated to post-singularly finite dynamical systems.
These groups encode the Julia set of a rational function from the point of view
of symbolic dynamics \cite{nekrashevych2005self}. Conversely, many classical examples of
self-similar groups with exotic geometric properties, such as the
Fabrykowski-Gupta \cite{FabrykowskiGupta} and the Basillica group \cite{GZ2002}, arise in a natural way as iterated
monodromy groups of
rational maps. 

Much of the study of symbolic dynamics of quadratic polynomials has been done
in terms of dynamic rays, as well as in terms of kneading sequences
\cite{bruinsymbolic,milnor1988iterated,thurston2009}, before Iterated Monodromy Groups were introduced as a new
and powerful tool \cite{nekrashevych2005self, Bartholdi2006}. The relationships
between these groups and kneading sequences were developed in
\cite{Bartholdi2008kneading}. 

This paper is a first in a series of papers that study of iterated monodromy
groups of entire functions. Here we focus on a particularly fundamental class of
functions, the exponential family, motivated by the well known strong analogy
between the combinatorics of quadratic polynomials and exponential maps (see e.g.\ \cite{DGH}).
Like polynomials, exponential maps have so far only been studied in terms of rays and
kneading sequences (see e.g.\ \cite{Zimmer2003}) resulting in a complete
classification in \cite{laubner2008}, based on \cite{HSS}.

In this paper, we introduce iterated monodromy groups for exponential maps and
compare them to self-similar groups defined just in terms of formal kneading
sequences. For an exponential map $f$, we show that the
iterated monodromy action of $f$ is conjugate to the self-similar group
action defined by the kneading sequence of $f$. For all kneading
sequences, we show that the obtained group is a left-orderable amenable
group that is residually solvable, but not residually finite. 

We give a short background in holomorphic dynamics in section 2, with a special focus on the exponential family. 
Next in section 3 we provide the algebraic and graph theoretic background to define the iterated monodromy group of a post-singularly
finite entire function. We give an explicit description of the iterated
monodromy group in terms of kneading automata in section 4, see
Theorem~\ref{thm:img}.
The structure of the orbital Schreier graphs is investigated in section 5, where
we show in Theorem~\ref{thm:schreier} that every component of the (reduced) orbital
Schreier graph is a tree with countably many ends. This result together with the
work in \cite{Reinkegroup} is then used in section 6, where we collect
group theoretic properties of the iterated monodromy groups of exponential functions, in particular amenability (see Theorem~\ref{thm:amenable}).

\emph{Acknowledgements.} We gratefully acknowledge support by the Advanced Grant HOLOGRAM by the European Research Council. Part of this research was
done during visits at Texas A\&M University and at UCLA\@.
We would like to thank our hosts, Volodymyr Nekrashevych and Mario Bonk, as well as the HOLOGRAM team, in particular Kostiantyn Drach, Dzmitry Dudko, Mikhail Hlushchanka, Wolf Jung, David Pfrang and Dierk Schleicher, for helpful discussions and comments.

\section{Dynamics of Exponential Maps}

\subsection{General entire dynamics}
We give a very short introduction into transcendental dynamics relevant to our needs, see \cite{schleicher2010dynamics} for a survey.
We start with the definition of a post-singularly finite entire function.
\begin{definition}
    Let $f \colon \C \rightarrow \C$ be an entire function. A \emph{critical value} is the image of a critical point, i.e.\ $f(c)$ where $f'(c) = 0$. An \emph{asymptotic value} is a limit $\lim_{t \to \infty}f(\gamma(t))$ where $\gamma \colon [0,\infty) \rightarrow \C$ is a path with $\lim_{t \to  \infty}\gamma(t) = \infty$.
        The set of singular values is defined as
        \begin{eqnarray*}
            \Sing(f) = \overline{\left\{\text{critical and asymptotic values}\right\}}
        \end{eqnarray*}
        and the set of post-singular values is $\PP(f) = \overline{\bigcup_{n\geq 0}f^n(\Sing(f))}$.
    The map $f$ is called \emph{post-singularly finite} if $\PP(f)$ is finite.
\end{definition}

The following lemma is the basis of our consideration:
\begin{lemma}[{\cite[Theorem~1.13]{schleicher2010dynamics}}]
  Let $f$ be an entire function. Then $f$ restricts to an unbranched covering from $\C \setminus f^{-1}(\Sing(f))$ to $\C \setminus \Sing(f)$.
	\label{lem:covering}
  \qedcited
\end{lemma}
In fact, an alternative definition of $\Sing(f)$ is that $\Sing(f)$ is the smallest closed subset $S$ such that $f$ restricts to an unbranched covering over $\C \setminus S$. As $\PP(f)$ is a closed and contains $\Sing(f)$, we see that that $f$ also restricts to an unbranched covering from $\C \setminus f^{-1}(\PP(f))$ to $\C \setminus \PP(f)$. As $\PP(f)$ is forward invariant, we have that $\PP(f) \subset f^{-1}(\PP(f)) \subset f^{-2}(\PP(f)) \subset \dots$
is an increasing chain of closed subsets. From this we can show by induction that $f^n$ restricts to an unbranched covering from
$\C \setminus f^{-n}(\PP(f)) \rightarrow \C \setminus \PP(f)$, using the fact that compositions of coverings of manifolds are again coverings.

The \emph{escaping set} $\Esc(f)$ is the set of points which escape to
infinity under the iteration of $f$, i.e.\
\begin{eqnarray*}
  \Esc(f) = \{z \colon \lim_{n \rightarrow \infty} f^n(z) = \infty\}\text{.}
\end{eqnarray*}

\begin{definition}
  A \emph{dynamic ray} is a maximal injective curve $\gamma \colon (0,\infty)
  \rightarrow \Esc(f)$ with $\gamma(t) \rightarrow \infty$ as
  $t\rightarrow \infty$.

  We say that $\gamma$ \emph{lands}
  at $a$ if $\gamma(t) \rightarrow a$ for $t \rightarrow 0$.
\end{definition}
We should note that this definition is not the precise standard definition given in \cite{schleicher2010dynamics},
however, it is appropriate in the study of post-singularly finite exponential maps as done in \cite{laubner2008}.
We will only use dynamic rays for exponential maps, so this is not an issue for us.

\subsection{Combinatorics of exponential maps}
The \emph{exponential family} is the family of functions $E_\lambda(z) = \lambda
\exp(z)$ for $\lambda \in \C^* \coloneqq \C \setminus \left\{ 0
\right\}$. The only singular value of $\lambda \exp(z)$ is $0$. It is the limiting value along the negative
real axis. It is also an omitted value, so for the exponential family,
Lemma~\ref{lem:covering}
specialized to the well-known fact that every function in
the exponential family is a covering from $\C$ to $\C^*$.

This is in
fact a universal covering, and the group of deck transformations are
given by translations of $2 \pi i$. In the following, we will often
consider collections which form a free orbit under translations with
multiplies of $2 \pi i$. A prime example is the set of preimages
${E^{-1}_\lambda}(z)$ of any
point $z  \in \C^*$.  As $\Sing(E_\lambda(z)) = \left\{ 0 \right\}$, we have
$\PP(E_\lambda(z)) = \left\{ E^n_\lambda(0) \colon n \geq 0 \right\} =
\left\{ 0, \lambda, E_\lambda(\lambda), \dots \right\}$. 

In this subsection, $f$ will always denote a post-singularly finite function in the
exponential family. In this setting, $0$ is a strictly preperiodic point, as it
is an omitted value and has finite forward orbit. We denote the preperiod of $0$
as $k$ and the period of $0$ as $p$, so $\PP(f) = \{0, f(0), \dots f^{k+p-1}(0)\}$
with $f^{k+p}(0) = f^k(0)$.

The dynamics of post-singularly finite exponential maps can be studied via dynamical rays, as seen in the following theorem:
\begin{theorem}[\cite{Zimmer2003}]
  Let $f(z) = \lambda \exp(z)$ be a post-singularly finite function in
  the exponential family. Then there is a dynamic ray landing at 0 which is
  preperiodic.
  \label{thm:ex-dynamicalray}
  \qedcited
\end{theorem}
We collect some facts about dynamic rays of exponential maps that are all discussed in \cite{Zimmer2003,laubner2008}.
\begin{fact}
  \begin{enumerate}
  \item Two different dynamic rays do not intersect, but they might land at
    the same point.
  \item The preimage of a dynamic ray is a family of dynamic rays forming a
		free orbit under translations with multiplies of $2 \pi i$.
  \item If $\gamma$ lands at
    $a$, then for every $b \in f^{-1}(a)$  there is a unique preimage component of
    $\gamma$ landing at $b$.
  \item If $\gamma$ lands at $0$, then all preimage components separate the
    plane, the connected components of $\C \setminus f^{-1}(a)$ also
		form a free orbit under  
    translations with multiples of $2 \pi i$.
  \end{enumerate}
  \qedcited
\end{fact}
\begin{definition}
  A \emph{ray spider} is a family $\Spider = (\gamma_a)_{a \in \PP(f)}$ such
  that $\gamma_a$ is a dynamic ray landing at $a$ for each $a \in \PP(f)$.
\end{definition}
\begin{remark}
In this definition, we do not require any invariance properties. 

Our notion of a ray spider is a special case of the general notion of
spiders given in \cite{Zimmer2003}.
By Theorem~\ref{thm:ex-dynamicalray}, there exists a ray spider: if
$\gamma$ is a dynamic ray landing at $0$,
then $\gamma_{f^i(0)} = f^i (\gamma), 0 \leq i < k + p$ is a ray spider.
This spider is not necessarily forward invariant, as it might happen that
$f^k(\gamma) \not=f^{k+p}(\gamma)$ (the period of the rays may be a multiple of the period of the landing point). This is not an issue in our construction as
we will consider the family of pullbacks of a given spider.
\end{remark}

\begin{definition}
  Let $\Spider = (\gamma_a)_{a \in \PP(f)}$ be a ray spider. The
  \emph{pullback} of $\Spider$ is the ray spider $(\tilde \gamma_a)$ where
  $\tilde \gamma_a$ is the unique preimage of $\gamma_{f(a)}$ landing at $a$.

  The \emph{dynamical partition} associated to $\Spider$ is the partition of $\C
  \setminus f^{-1}(\gamma_0)$ into its connected components. We denote the
  connected component of $0$ by $\Ipart_0$ and define $\Ipart_n =
	\Ipart_0 +2 \pi i n$. Note that the dynamical partition only
  depends on the ray landing at 0.

  The kneading sequence of $f$ is the sequence $(k_n)_{n \in \N}$ so
  that $f^n(0) \in \Ipart_{k_n}$. The kneading sequence is in fact
	independent of $\Spider$, see \cite{laubner2008} for a more detailed
	discussion. 
\end{definition}
\begin{example}
  Let $k \in \Z \setminus {0}$, and consider $f(z) = 2 k \pi i \exp(z)$. For
  this map, $0$
  is mapped to $2k \pi i$, which is a fixed point of $f$. Hence $f$ is
  post-singularly finite with $\PP(f) = \{0, 2k \pi i\}$. Let $\gamma_0$ be a
  dynamic ray landing at 0, and let $\Ipart$ be the associated dynamical
  partition. Then $0 \in \Ipart_0$ by definition of $\Ipart_0$ and $2 k \pi i
  \in \Ipart_k = \Ipart_0 + 2 k \pi i$, so the kneading sequence of $f$ is
  $0\overline k$.
\end{example}

\section{Iterated Monodromy Groups}

\subsection{The dynamical preimage tree $\T$}

Let $f$ be a post-singularly finite entire function and $t \in \C \setminus \PP(f)$.
    \begin{definition}
      Choose a base point $t \in \C \setminus \PP(f)$. Let $L_n \coloneqq f^{-n}(t)$ be the
      preimage of $t$ under the $n$-th iterate of $f$.

      The \emph{dynamical preimage tree}
      $\T$ is a rooted tree with vertex set $\bigsqcup_{n \geq 0}L_n$  (where $\bigsqcup$
      denotes disjoint union) and edges $w \rightarrow f(w)$ for $w \in L_{n+1}, f(w) \in L_{n}$. Its root is $t$.
    \end{definition}
    The dynamical preimage tree is always a regular rooted tree, i.e.\ all vertices have the same number of children. For polynomials, this number is
    the degree of the polynomial. For transcendental entire functions, every vertex has countably infinite many children.
    We will show in subsection~\ref{sss:trees} that for postsingularly finite exponential maps, the dynamical preimage tree has an extra regularity
    based on the periodicity of the exponential map.
    \begin{figure}
      \centering
      \begin{tikzpicture}[node distance=1.7cm]
        \tikzstyle{level 1}=[sibling distance=50pt] \tikzstyle{level 2}=[sibling
        distance=10pt] \node (root) {$t=L_0$} [grow'=up] child[dotted] {node
          (l1) {$f^{-1}(t)=L_1$}} child { child[dotted] {node{}} child {} child
          {} child {} child[dotted] {}} child { child[dotted] {} child {} child
          {} child {} child[dotted] {}} child { child[dotted] {} child {} child
          {} child {} child[dotted] {}} child[dotted] {}; \node (l2) [above
        of=l1] {$f^{-2}(t)=L_2$};
      \end{tikzpicture}
      \caption{Dynamical preimage tree}
      \label{fig:tree}
    \end{figure}
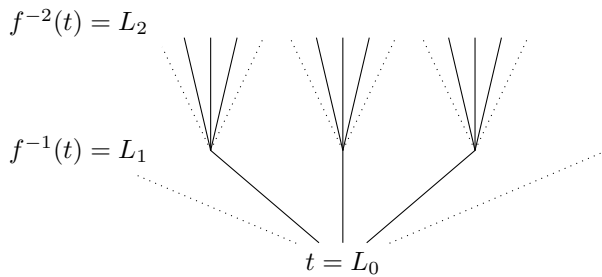
\subsection{Iterated Monodromy Action}
Each level of $\T$ is the preimage of $t$ under a covering map, namely $f^n
\colon \C \setminus f^{-n}(\PP(f)) \rightarrow \C \setminus \PP(f)$. Hence
$\pi_{1}(\C \setminus \PP(f), t)$ acts on $L_n$ via path lifting:
if $\gamma \colon [0,1] \rightarrow \C \setminus \PP(f)$ is a loop based on $t$
and $v \in L_n$ is a $n$-th preimage of $t$, then there is a unique lift
$\gamma^v$ making the following diagram commute: 
\begin{equation*}
  \begin{tikzcd}[ampersand replacement=\&]
    \& \left(\C \setminus f^{-n}(\PP(f)), v \right) \ar[d,"f^n"] \\
    \left( \left[ 0,1 \right], 0 \right)\ar[r,"\gamma"] \ar[ru,"\gamma^v", dotted] \& \left(  \C \setminus \PP(f), t \right) \\
  \end{tikzcd}
\end{equation*}
So $\gamma^v(0) = v$, and $\gamma^v(1) \in L_n$ might be another $n$-th preimage.
We define $[\gamma](v) \coloneqq \gamma^v(1)$. Using the homotopy
lifting properties of coverings, we can see that this defines an action of
$\pi_{1}(\C \setminus \PP(f), t)$ on $L_n$. If $w \in L_{n+1}$ is a child of
$v$, then the following diagram commutes (by uniqueness of lifts):

\begin{equation*}
  \begin{tikzcd}[ampersand replacement=\&]
    \& \left(\C \setminus f^{-n-1}(\PP(f)), w \right) \ar[d,"f^n"] \\
    \& \left(\C \setminus f^{-n}(\PP(f)), v \right) \ar[d,"f^n"] \\
    \left( \left[ 0,1 \right], 0 \right)\ar[r,"\gamma"] \ar[ruu,"\gamma^w", dotted] \ar[ru,"\gamma^v", dotted] \& \left(  \C \setminus \PP(f), t \right) \\
  \end{tikzcd}
\end{equation*}
By commutativity of the diagram $f(\gamma^w(1)) = \gamma^v(1)$
so $[\gamma](w)$ is also a child of $[\gamma](v)$.
This means that actions on the levels are compatible and give rise to an action of
$\pi_{1}(\C \setminus \PP(f))$ on $\T$. This is the $\emph{iterated monodromy action}$.
\begin{definition}
 Let $f$ be a post-singularly finite  entire function, $t \in \C \setminus \PP(f)$.
 Let $\phi \colon \pi_{1}(\C \setminus \PP(f), t) \rightarrow \Aut(\T)$ be
 the group homomorphism induced by the iterated monodromy action. The iterated monodromy group of $f$ with base point $t$ is
 the image of $\phi$. By the first factor theorem we have
            \begin{eqnarray*}
                \IMG(f) \cong \faktor{\pi_{1}(\C \setminus \PP(f),t)}{\ker \phi}
            \end{eqnarray*}
\end{definition}
This definition depends a priori on the base point $t \in \C \setminus \PP(f)$.
For a different base point $t'$, every path from $t$ to $t'$ gives rise to an
isomorphism of preimage trees over $t$ and over $t'$, so we can identify the
groups up to inner automorphisms. 
See \cite[Proposition~5.1.2]{nekrashevych2005self} for a detailed discussion
in the rational case.
\subsection{$\Z$-regular rooted trees}
\label{sss:trees}
We use the following definition of rooted trees:
\begin{definition}
  A \emph{rooted tree} is a tuple $T=(V, E, r)$ such that $(V, E)$ forms a tree (with
  vertex set $V$ and edge set $E$) and $r \in V$, which we call the \emph{root} of $T$. 
  We endow $T$ with the
  unique orientation so that all vertices are reachable from the root, i.e.\
  for every vertex $v$, there is directed path from the root to $v$. 

  If
  $(v,w)$ is a directed edge for this orientation, we say that $w$ is a \emph{child}
  of $v$ and $v$ is the \emph{parent} of $w$. If $v$ has no children, we call it
  a \emph{leaf}.

  If $w$ is reachable form $v$, we say that $w$ is a descendant of $v$ and $v$
  is an ancestor of $w$. We denote by $T_v$ the rooted tree which is the induced
  subgraph on the set of descendants of $v$ together with $v$ as the new root.
  An end of a rooted tree $T$ is a sequence $v_n$ so that $v_0$ is the root of
  $T$ and $v_{n+1}$ is a child of $v_n$. We denote by $\partial T$ the set of
  ends of $T$.
\end{definition}
We will mainly consider countable infinite trees without leaves.  

In fact, $\partial T$ can be defined without fixing a root of $T$, one way is by
considering equivalence classes of geodesic rays, where two geodesic rays are
equivalent if they have a common tail. Given a root $r$ and a geodesic ray $\gamma$, there is
always a unique geodesic ray starting at $r$ equivalent to $\gamma$.
Also, $\partial T$ is a totally disconnected Hausdorff space with clopen
subset $\partial T_v \subset \partial T$. The topology is also
independent of the root. If $T$ is a locally finite tree without leaves,
then $\partial T$ is compact.
\begin{definition}
  A $\Z$\emph{-regular rooted tree} $T$ is a tuple $(V, E, r, \eta)$, where $(V, E, r)$ is
  a rooted tree and $\eta$ is a right $\Z$-action $\eta \colon V \setminus\{r\}
  \times \Z \rightarrow V \setminus\{r\}$
  such that for all vertices $v \in V$, the set of its children forms a free orbit under
  the action.

  Note that this implies that the root is fixed by the action, as it is the only
  vertex without a parent. Also the tree has no leaves, as the empty set is not
  a free orbit under a $\Z$-action.

  An isomorphism between $\Z$-regular rooted trees is a tree isomorphism which
  preserves the root and commutes with the additional right $\Z$-actions.
  We denote by $\Aut_\Z(T)$ the group of automorphisms of $T$ as a $\Z$-regular
  rooted tree. Every element of $\Aut_\Z(T)$ preserves the root of $T$ and acts by a translation on the first level.
  We denote by $\rho \colon \Aut_\Z(T) \rightarrow \Z$ the group homomorphism given by the first level action.
  The kernel of $\rho$ is the stabilizer of the first level, as every element of $\Aut_\Z(T)$ acts by translation,
  this is also the stabilizer of any vertex on the first level. For a vertex $v \in V$ and a subgroup $G \subset \Aut_\Z(T)$ we denote the stabilizer
  of v in $G$ by $\Stab_G(v)$. We denote the stabilizer of the first level as $\Stab_G$. 
\end{definition}
Note that $\Aut_\Z(T)$ also acts on $\partial T$. This action is an fact
faithful,
as every vertex is part of a sequence defining an end.
\begin{example}
  The standard $\Z$-regular tree has as vertex set $\Z^*$, the set of finite
  words in $\Z$. Its root is the empty word $\emptyset$. Its edges are all pairs
  of the form $(v, vn)$ for $v \in \Z^*, n \in \Z$ (here $vn$ denotes the word
  $v$ concatenated with the letter $n$). So for each vertex $v$, the set of it
  children are all words obtained by concatenating one letter to it.
  Also, the set of ends can be identified with the set of right-infinite
  words, which we denote by $\Z^\omega$.
  
  The right action is given by
  \begin{eqnarray*}
    \eta(vn,m) = v(n+m)\text{.}
  \end{eqnarray*}
  So the action is by translation on the last letter.
  By abuse of notation, we will denote the standard $\Z$-regular tree also by $\Z^*$.
\end{example}
The subgroups of $\Aut_\Z(\Z^*)$ were studied in
\cite{oliynyk2010groups} under the name of \emph{ZC-groups}.
Note that if $T$ is a $\Z$-regular rooted tree and $v$ is a vertex of $T$, then
$T_v$ is also a $\Z$-regular rooted tree. However, in general we have no
canonical choice of an isomorphism between $T$ and $T_v$. This is different for
the standard $\Z$-regular tree:
\begin{definition}
  For $g \in \Aut_\Z(\Z^*), v \in \Z^*$ let $g_{|v}$ denote the unique element in
  $\Aut_\Z(\Z^*)$ such that $g(vw) = g(v)g_{|v}(w)$. We say that $g_{|v}$ is the \emph{section}
  of $g$ at $v$.
\end{definition}
We will use the following set of easily verifiable cocycle equations:
\begin{eqnarray}
  (g_{|v})_{|w} = g_{|vw} \\
  (gh)_{|v} = g_{|h(v)} h_{|v}
  \label{eqn:cocycle}
\end{eqnarray}

We say that $g \in \Aut_\Z(\Z^*)$ is of finite activity on level $n$ if the set $\{v \in \Z^n \colon
g_{|v} \not= \Id\}$ is finite. We define $\aut_\Z(\Z^*)$ as the group of
automorphisms which have finite activity on every level. We will many work with subgroups of $\aut_\Z(\Z^*)$.
As we work with an infinite alphabet, we have to take care for the wreath recursion. The wreath recursion for $\aut_\Z(\Z^*)$ is
\begin{eqnarray*}
  \aut_\Z(\Z) &\cong& \left(\bigoplus_{x \in \Z} \aut_\Z(\Z)\right) \rtimes \Z \\
  g &\mapsto& (x \mapsto g_{|x}, \rho(g))
\end{eqnarray*}
We say a subgroup $G \subset \Aut_\Z(\Z^*)$ is self-similar if $g_{|v} \in G$ for all $g \in G$ and $v \in \Z^*$.
A subgroup $G \subset \Aut_\Z(\Z^*)$ is self-replicating if for all $v \in \Z^*$ and $g \in G$ there exists an $h \in G$ with $h_{|v} = g$. It is easy to see that is is enough to check this on the first level.

  \begin{lemma}
    Let $f$ be a post-singularly finite exponential function, $t \in \C \setminus \PP(f)$. Then the dynamical preimage tree
    of $f$ with base point $t$ is a $\Z$-regular tree and $\IMG(f)$ is a subgroup of $\Aut_\Z(\T)$ 
    \label{lem:expzregular}
  \end{lemma}
  \begin{proof}
    The $\Z$-regular structure is given by translation by multiples of $2\pi i$. As two complex numbers have the same value
    under the exponential map if and only if they differ by a multiple of $2 \pi i$, it is clear that this really defines a
    $\Z$-regular structure. Also, if $w$ is an $n$-th preimage of $t$, and $\gamma$ is a loop on $t$, for the lift
    $\gamma^w$, the $2\pi i$ translate of $\gamma^w$ is also a lift of $\gamma$ by the $2 \pi i$ periodicity of $f^n$. This shows
    that the iterated monodromy action commutes which the $\Z$ action given by the $\Z$-regular structure, so $\IMG(f) \subset \Aut_\Z(\T)$.
  \end{proof}

\section{Combinatorial description}
\subsection{Automata}
\begin{definition}
 An automaton $\A$ is a map $\tau \colon Q \times X \rightarrow X \times Q$.
 We call Q the \emph{state set} and X the \emph{alphabet}. We will write the
 components of $\tau(a,x)$ often as $(a(x), a{|x})$. Here $a(x) \in X$ is called
 the image of $x$ under $a$, and $a{|x}$ is the restriction of $a$ at $x$.

 A $\emph{group automaton}$ is an automaton such that for all $a \in Q$, the map
 $x \mapsto a(x)$ is a bijection on $Q$. If the alphabet is $\Z$, that automaton
 is a $\Z$-automaton if for all $a \in Q$, the map $x \mapsto a(x)$ is a
 translation on $\Z$, i.e.\ equal to the map $x \mapsto x + n$ for some $n\in \Z$.
\end{definition}
We will only consider automata which have a distinguished identity state $\Id$, i.e.\ a state
such that $\tau(\Id,x) = (x, \Id)$ for all $x \in X$.
 We can draw automata using Moore diagram. As vertices we take the state set
 $Q$, and if $\tau(a, x) = (y, b)$, we draw an edge from $a$ to $b$ labeled $x|y$.
 Here is an example of a Moore diagram, of the so-called binary adding machine.
 
                
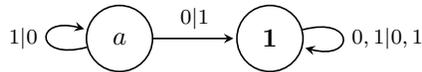
\begin{equation*}
                    \begin{tikzpicture}[->,>=stealth,shorten >=1pt,auto,node distance=2cm,semithick]
                        \node[state] (A) {$a$};
                        \node[state] (I) [right of=A] {$\Id$};
                        \tikzstyle{every node}=[font=\footnotesize]
                        \path (A) edge node {$0|1$} (I)
                        (A) edge [loop left] node {$1|0$} (A)
                        (I) edge [loop right] node {$0,1|0,1$} (I);
                    \end{tikzpicture}
                \end{equation*}
\begin{definition}
  Let $\A$ be an automaton given by $\tau \colon Q \times X
  \rightarrow X \times Q$. We extend $\tau$ to a map $Q \times X^* \rightarrow
  X^* \times Q$ recursively via
  \begin{equation*}
   \tau(a, xv) = (a(x)a{|x}(v), a{|x}{|v}) 
  \end{equation*}
\end{definition}
If $\A$ is a group automaton, then for each $a \in A$, the extended map $X^*
\rightarrow X^*$ induces a tree automorphism of the regular $X$-tree. If $\A$ is
a $\Z$-automaton, it is a automorphism preserving the regular $\Z$-tree
structure.
\subsection{Kneading automata}
\begin{definition}
  Given two words $x_1\dots x_k, y_1 \dots y_p \in \Z^*$ with $x_k \not= y_p$ the automaton $\Kaut(x_1\dots x_k, y_1 \dots y_p)$
  has alphabet $\Z$ and states $a_1 \dots a_k, b_1 \dots b_p$ (and the identity
  state $\Id$) and the following transition function:
  \begin{eqnarray*}
    \tau(a_1, z) &=& (z+1, \Id) \\
    \tau(a_{i+1},x_i) &=& (x_i, a_i) \\
    \tau(b_1,x_k) &=& (x_k, a_k) \\
    \tau(b_1,y_p) &=& (y_p, b_p) \\
    \tau(b_{i+1},y_i) &=& (y_i, b_i) \\
    \tau(q,z) &=& (z, \Id) \text{ for all other cases.}
  \end{eqnarray*}
  \label{def:kneading}
\end{definition}
We note that $\Kaut(x_1\dots x_k, y_1 \dots y_p)$ is a $\Z$-automaton, indeed
$a_1$ acts on $\Z$ by the translation by one, and all other states act on $\Z$
as the identity. Figure~\ref{fig:moore} shows a reduced Moore diagram of $\Kaut(x_1\dots
x_k, y_1 \dots y_p)$, where labels with only one letter $z$ are abbreviations
for the label $z|z$ and all trivial arrows ending in the identity state
have been omitted.
\begin{example}
  The automaton $\Kaut(0,k)$ with $k \in \Z \setminus \left\{ 0 \right\}$ has the following (non-reduced) Moore diagram: 
                
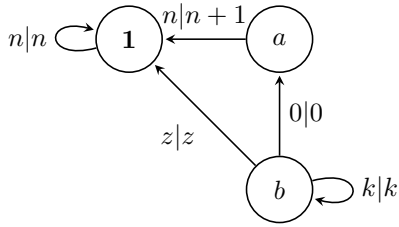
\begin{equation*}
                    \begin{tikzpicture}[->,>=stealth,shorten >=1pt,auto,node
                      distance=2cm,semithick]
                      \node[state] (id) {$\Id$};
                      \node[state] (a) [right of=id] {$a$};
                      \node[state] (b) [below of=a] {$b$};
                      \path (id) edge [loop left] node {$n|n$} (id);
                      \path (a) edge node [swap] {$n|n+1$} (id)
                            (b) edge node [swap] {$0|0$} (a);
                      \path (b) edge [loop right] node {$k|k$} (b);
                      \path (b) edge node {$z|z$} (id);
                    \end{tikzpicture}
                \end{equation*}
                Here $n$ stands for any element of $\Z$, and $z$ for any element of $\Z \setminus \left\{ 0, k \right\}$.
  \label{exa:Moore0k}
\end{example}
\begin{remark}
  We see that every non-trivial state has exactly one edge ending in it, so
  for every non-trivial state there is a unique left-infinite path ending in it.
  This implies that $\Kaut(x_1\dots x_k, y_1 \dots y_p)$ is a bounded activity
  automaton in the sense of \cite{sidki2004finite}: For any length $m$, there
  are $n+k$ paths of length $m$ ending in a non-trivial state in the Moore diagram, so for any $q$,
  the set $\left\{ v \in \Z^m \colon q|v \not= \Id \right\}$ has cardinality bounded
  by $n+k$. 
  \label{rem:boundedactivity}
\end{remark}
We denote by $\Kgroup(x_1\dots x_k, y_1 \dots y_p)$ the group of
          automorphisms of $\Z^*$ generated by
          $\Kaut(x_1 \dots x_k, y_1 \dots y_p)$.
          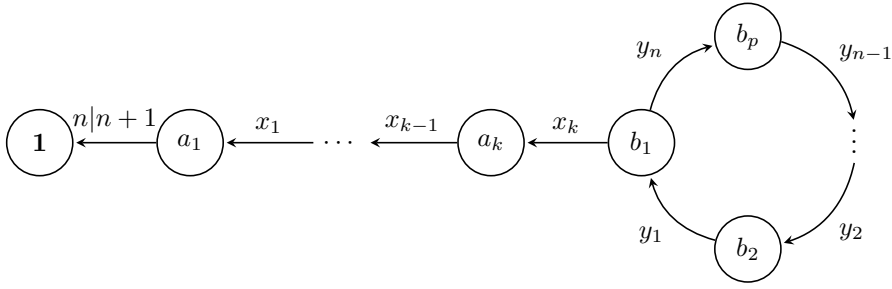
\begin{figure}
            \centering
            \begin{tikzpicture}[->,>=stealth,shorten >=1pt,auto,node
              distance=2cm,semithick]
              \node[state] (id) {$\Id$};
              \node[state] (a-1) [right of=id] {$a_1$};
              \node (a-2) [right of=a-1] {$\cdots$};
              \node[state] (a-k) [right of=a-2] {$a_k$};
              \node[state] (b-1) [right of=a-k] {$b_1$};
              \node[state] (b-2) [below right of=b-1] {$b_2$};
              \node (b-d) [above right of=b-2] {$\rvdots$};
              \node[state] (b-n) [above left of=b-d] {$b_p$};
              \path (a-1) edge node [swap] {$n|n+1$} (id)
                    (a-2) edge node [swap] {$x_1$} (a-1)
                    (a-k) edge node [swap] {$x_{k-1}$} (a-2)
                    (b-1) edge node [swap] {$x_{k}$} (a-k);
              \path[ bend left] (b-2) edge node {$y_1$} (b-1)
                                (b-d.south) edge node {$y_2$} (b-2)
                                (b-n) edge node {$y_{n-1}$} (b-d.north)
                                (b-1) edge node {$y_n$} (b-n);
            \end{tikzpicture}
            \caption{Moore diagram of kneading automata}
            \label{fig:moore}
          \end{figure}
      \begin{theorem}
        Let $f$ be a post-singularly finite exponential function with kneading
        sequence $x_1\dots x_k\overline{y_1\dots y_p} \in \Z^\omega$. Then the
        iterated monodromy action of $f$ is conjugate to the action of $\Kgroup(x_1
        \dots x_k, y_1 \dots y_p)$ on $\Z^*$.
        \label{thm:img}
      \end{theorem}
      In particular, for functions of the form $2 \pi i k \exp(z)$ with $k \in \Z \setminus \left\{ 0 \right\}$,
      the iterated monodromy action is conjugate to the action of the automata group $\Kgroup(0,k)$ discussed in Example~\ref{exa:Moore0k}.
      \begin{proof}
We choose a ray spider $\Spider_0$ for $f$ and consider the sequence
$\Spider_n$, where $\Spider_{n+1}$ is the pullback of $\Spider_{n}$. We denote
by $\gamma_{z,n}$ the ray in $\Spider_n$ landing at $z$, also let $\Ipart_{*,n}$ be
the dynamical partition induced by $\Spider_n$. 
Choose a base point $t \in \C \setminus \bigcup_{z \in \PP(f)} \gamma_{z,0}$.
We recursively define an isomorphism between the dynamical preimage tree $\T$ and
the standard $\Z$-tree $\Z^*$. We send the root $t$ to the empty word
$\emptyset$.
Suppose we already defined the bijection on $L_n \subset \T$, and let $w \in
L_n$ be mapped to $v \in \Z^n$. Then for the dynamical partition $\Ipart_{*,n}$,
there is exactly one child of $w$ in each component. We send the child lying in
$\Ipart_{m,n}$ to $vm$.

By construction, this defines an isomorphism of $\Z$-trees.

The complement of each ray spider is a simply connected domain. For two points
$w_1, w_2 \in \C \setminus \bigcup_{z \in \PP(f)} \gamma_{z,n}$ let
$g_n(w_1,w_2)$ be a path from $w_1$ to $w_2$ crossing no ray of $\Spider_n$
and let $g_{z,n}(w_1, w_2)$ be a path from $w_1$ to $w_2$ crossing only the ray
of $\gamma_{z,n}$ once in a positive sense (so that $g_n(w_1,w_2)$ composed with
$g_{z,n}(w_1,w_2)$ has winding number $1$ around $z$) and no other ray of
$\Spider_n$. The homotopy classes of $g_n$ are well defined in the fundamental
groupoid $\Pi_1(\C \setminus \PP(f))$.
Let us investigate the lifting behavior of these homotopy classes: let $w, w'
\in \C \setminus \bigcup_{z \in \PP(f)} \gamma_{z,n}$
and let $v \in
f^{-1}(w)$. Let $g^v_{n}(w, w')$ (or $g^v_{z,n}(w,w')$) denote the
lift of $g_{n}(w,w')$ (respectively $g_{z,n}(w,w')$).
Then $g^v_{n}(w,w')$ is a path in $\C$ meeting no preimage of $\gamma_{z,n}$
for $z \in \PP(f)$. Let $v'$ be the preimage of $w'$ in the same component of
$\Ipart_{*,n}$ as $v$. Then $g^v_{n}(w,w')$ must be homotopic to $g_{n+1}(v,v')$.
Similarly, $g^v_{0,n}(w,w)$ is a path which doesn't cross any ray of
$\Spider_{n+1}$, and as $g_{0,n}(w,w)$ has winding number 1 around 0, the lift
$g^v_{0,n}(w,w)$ must end in $v+ 2 \pi i$. Hence $g^v_{0,n}(w,w) \cong g_{n+1}(v,v+1)$
and by composition $g^v_{0,n}(w,w') \cong g_{n+1}(v,v'+1)$.
Let $z \in \PP(f) \setminus {0}$. Then $g^v_{z,n}(w,w')$ crosses no boundary of
$\Ipart_{*,n}$, so it must end in $w'$. If $\tilde z$, the preimage of $z$ in the same component
of $\Ipart_{*,n}$ is in $\PP(f)$, then $g^v_{z,n}(w,w') \cong g_{\tilde z,
  n+1}(v,v')$, otherwise $g^v_{z,n}(w,w') \cong g_{n+1}(v,v')$.

Now $\pi_1(\C \setminus \PP(f), t)$ is freely generated by $(g_{z,0}(t,t))_{z\in
  \PP(f)})$. Numerate $\PP(f)$ by $z_1 = 0, z_{i+1} = f(z_i), 1 \leq i \leq
n+k-1$. We claim that the group homomorphism given by
\begin{eqnarray}
  g_{z_i,0}(t,t) \mapsto a_i, 1 \leq i \leq k \\
  g_{z_i,0}(t,t) \mapsto b_{i-k} k + 1 \leq i \leq n+k 
\end{eqnarray}
conjugates the iterated monodromy action of $f$ to the action of
$\Kgroup(x_1\dots x_k, y_1 \dots y_p)$. 
This follows from the pullback behavior.

      \end{proof}
\section{Schreier Graphs}
For this section, we fix $x_1\dots x_k, y_1 \dots y_p$ with $x_k \not= y_p$.
We will give a combinatorial description of the action of
$\Kgroup(x_1 \dots x_k, y_1 \dots y_p)$ on the standard $\Z$-tree $\Z^*$.
We will work in this section with the generating set $S \coloneqq \left\{ a_1, \dots, a_k,
b_1, \dots, b_p  \right\}$ of $\Kgroup(x_1 \dots x_k, y_1 \dots y_p)$.
\begin{definition}
 Let $n \in \N$. The $n$-th level \emph{Schreier graph} has vertex set $\Z^n$
 and edges $v \rightarrow s(v)$ for $v \in \Z^n, s \in S$. The \emph{orbital Schreier graph} $\Gamma_\infty$ has the ends of
 the standard $\Z$-tree as vertex set (which can be identified with $\Z^\omega$) and 
 also has edges $v \rightarrow s(v)$ for $v \in \Z^\omega, s \in S$.
 
 The \emph{reduced Schreier graph} $\overline{\Gamma}_n$ and reduced orbital Schreier graph
 $\overline{\Gamma}_\infty$ are obtained by deleting
 all loops of $\Gamma_n$ respectively $\Gamma_\infty$.
\end{definition}
Let $w_m \in \Z^m$ be the reverse of the length $m$ prefix of $x_1\dots
x_k\overline{y_1\dots y_p}$. In the Moore diagram in Figure~\ref{fig:moore}, we
see that $w_m$ is the concatenation of the labels of the unique path $p$ of length m
ending in $a_1$. Let $c_m$ be the starting state of $p$ (so $c_m = a_m$ for $m
\leq k$, and $c_m = b_{m'}$ for appropriate $m'$ otherwise). Then $c_m|w_m =
a_1$ and $s|v \not= a_1$ for all other pairs of a state $s$ and $v \in \Z^m$.
As $a_1$ is the only state
 which acts non-trivially on the first level, we have
 \begin{eqnarray*}
   c_m|w_m(i) &=& i+1 \\
     s|v(i) &=& i \text{ for other pairs.}
 \end{eqnarray*}
Since additionally $a_1$ only restricts to the identity state, we also have that
if $s(v) = w$ with $v \not=w \in \Z^m$ for some state $s$, then $s(vi) = wi$ for
all $i \in \Z$. In fact $v$ and $w$ must differ in exactly one position.

This discussion can be summarized in the following lemma:
 \begin{lemma}
  The Schreier graph $\overline{\Gamma}_{m+1}$ can be obtained from
  $\overline{\Gamma}_{m}$ in the following way:
  take as vertex set $vx$ where $v \in \Z^m, x \in \Z$. 
  For edges we have the following two construction rules:
  \begin{itemize}
    \item $v \rightarrow v'$  edge in $\overline{\Gamma}_{m}$ \\ $\Rightarrow vi \rightarrow v'i$
      is an edge in $\overline{\Gamma}_{m+1}$ for all $i \in \Z$.
    \item $w_m i \rightarrow w_m (i+1)$ for
      all $i \in \Z$.
    \end{itemize}
 \begin{figure}
   \centering
    \begin{tikzpicture}[->,>=stealth,shorten >=1pt,auto,node distance=3cm,semithick]
        \tikzstyle{treesketch}=[ellipse,draw,dotted,minimum height=4.5cm, minimum width=2cm]
        \node (A) {$\cdots$};
        \node[state] (B) [right of=A] {$w_m0$};
        \node[state] (C) [right of=B] {$w_m1$};
        \node[state] (D) [right of=C] {$w_m2$};
        \node (E) [right of=D,] {$\cdots$};
        \node[treesketch] (Bt) [label={below:$\overline{\Gamma}_m0$}] at (B) {};
        \node[treesketch] (Ct) [label={below:$\overline{\Gamma}_m1$}] at (C) {};
        \node[treesketch] (TD) [label={below:$\overline{\Gamma}_m2$}] at (D) {};
        \tikzstyle{every node}=[node distance=2cm]
        \node (Ba) [above of=B] {$\rvdots$};
        \node (Ca) [above of=C] {$\rvdots$};
        \node (Da) [above of=D] {$\rvdots$};
        \node (Bb) [below of=B] {$\rvdots$};
        \node (Cb) [below of=C] {$\rvdots$};
        \node (Db) [below of=D] {$\rvdots$};
        \path
              (Ba) edge node {} (B)
              (B) edge node {} (Bb)
              (Ca) edge node {} (C)
              (C) edge node {} (Cb)
              (Da) edge node {} (D)
              (D) edge node {} (Db)
              (A) edge node {$c_m$} (B)
              (B) edge node {$c_m$} (C)
              (C) edge node {$c_m$} (D)
              (D) edge node {$c_m$} (E);
    \end{tikzpicture}
   \caption{Inductive construction of Schreier graphs}
   \label{fig:induct}
 \end{figure}
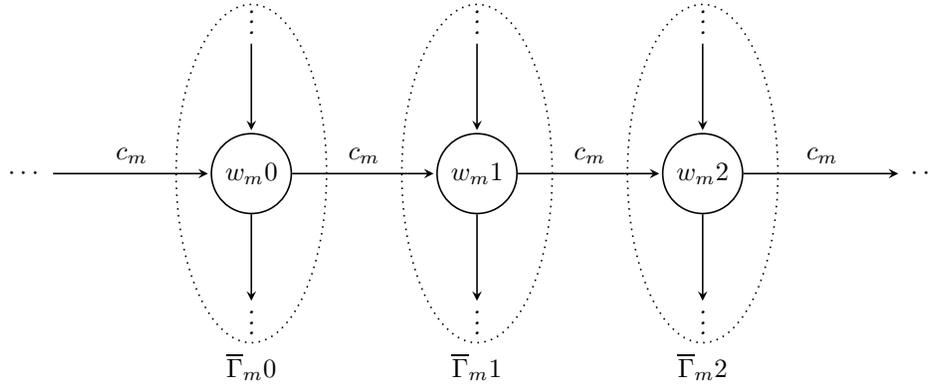
  See Figure~\ref{fig:induct} for a visualization of the construction rules.
	\label{lem:induct}
  \qed
 \end{lemma}
 \begin{example}
   We can use this construction to produce the first few $\overline{\Gamma}_{m}$ for the group
   $\Kgroup(0,1)$. As in Example~\ref{exa:Moore0k}, we name the generators $a$ and $b$ instead of
   $a_1$ and $b_1$. Note that $a$ acts by translation on the first level, and $b$ acts trivially on the first level, 
   so $\overline{\Gamma}_{1}$ is just a bi-infinite line. 
   To use the construction rule, we note that $w_1 = 0$, so we obtain
	 $\overline{\Gamma}_{2}$ as a comb in Figure~\ref{fig:sndlevel}.
	 \begin{figure}
		 \centering
		 \begin{tikzpicture}[->,>=stealth,shorten >=1pt,auto,node distance=2cm,semithick]
			 \node (Z-Z) {$\ddots$};
			 \foreach \nameB/\prevB/\labB in {A/Z/\overline{1},B/A/0,C/B/1,D/C/2}
			 \node (Z-\nameB) [below of = Z-\prevB] {$\cdots$};
			 \foreach \nameA/\prevA/\labA in {A/Z/\overline{1},B/A/0,C/B/1,D/C/2}
			 {
				 \node (\nameA-Z) [right of=\prevA-Z] {$\vdots$};
				 \foreach \nameB/\prevB/\labB in {A/Z/\overline{1},B/A/0,C/B/1,D/C/2}
				 \node[state] (\nameA-\nameB) [below of=\nameA-\prevB] {$\labA,\labB$};
				 \node (\nameA-E) [below of=\nameA-D] {$\vdots$};
			 }
			 \node (E-Z) [right of=D-Z] {$\iddots$};
			 \node (Z-E) [below of=Z-D] {$\iddots$};
			 \node (E-E) [right of=D-E] {$\ddots$};
			 \foreach \nameB/\prevB/\labB in {A/Z/\overline{1},B/A/0,C/B/1,D/C/2}
			 \node (E-\nameB) [below of = E-\prevB] {$\cdots$};
			 \tikzstyle{every node}=[font=\footnotesize]
			 \foreach \toA/\fromA in {A/Z,B/A,C/B,D/C,E/D}
			 \foreach \posB in {A,B,C,D}
			 \path[red]
			 (\fromA-\posB) edge node {} (\toA-\posB);

			 \foreach \toB/\fromB in {A/Z,B/A,C/B,D/C,E/D}
			 \path[blue]
			 (B-\fromB) edge  node {} (B-\toB);
			 \path[blue]
			 (C-B) edge [loop above] node {\color{red}$a$} (C-B)
			 (C-C) edge [loop above] node {$b$} (C-C);
		 \end{tikzpicture}
     \caption{Second-level Schreier-Graph of $\Kgroup(0,1)$ (where $\overline{1}=-1$ for notational convenience)}
   \label{fig:sndlevel}
	 \end{figure}
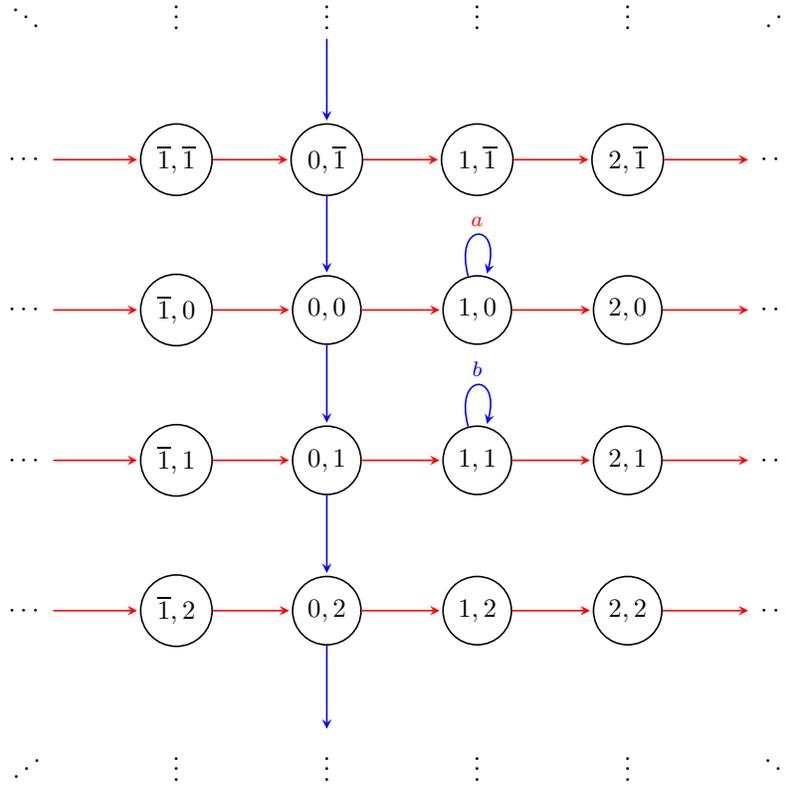
The loops at $1,0$ and $1,1$ are of course not present in the reduced Schreier graph,
but we did include them here for they are the loops which ``split up'' in the further generations:
as $b$ restricts to $a$ at $1,0$, we obtain $\overline{\Gamma}_{3}$ by connecting $\Z$ many copies of $\overline{\Gamma}_{2}$ by an bi-infinite line
going through the copies of $1,0$.
 \end{example}
 With this inductive description we can prove the following:
\begin{lemma}
  For all $m \in \N$, the reduced Schreier graph $\overline{\Gamma}_m$ is
  a tree with countably (or finitely) many ends.
\end{lemma}

\begin{proof}
  We do induction over m. For $m = 1$, the Schreier graph $\overline{\Gamma}_1$ is a bi-infinite line, so it is in
  particular a tree with finitely many ends. Now by
	Lemma~\ref{lem:induct}, $\overline{\Gamma}_{m+1}$ is
  the union of countably many copies of $\overline{\Gamma}_m$ and a bi-infinite
  line intersecting each copy in one point. So it is again a tree.
	We claim that have the following inductive description of the space of ends:
	\begin{equation}
		\partial \overline{\Gamma}_{m+1} \cong \Z \times \partial
		\overline{\Gamma}_{m} \cup \left\{ -\infty, +\infty \right\}
		\label{eq:endinduct}
	\end{equation}
  Here the right hand space is a compactification of $\Z \times
	\partial \overline{\Gamma}_{m}$, where $-\infty$ has the open sets
	$U_{<n} \coloneqq \left\{z \in \Z \colon z < n\right\} \times \partial
	\overline{\Gamma}_{m} \cup \left\{-\infty\right\}$ as neighborhood
	basis, and similarly $+\infty$ has the open sets $U_{>n} \coloneqq
	\left\{z \in \Z \colon z > n\right\} \times \partial
	\overline{\Gamma}_{m} \cup \left\{+\infty\right\}$ as neighborhood
	basis.
	The identification in \eqref{eq:endinduct} works as follows: we take
	$w_m$ as our root of $\overline{\Gamma}_{m}$ and $w_m0$ as the root of  
	$\partial \overline{\Gamma}_{m+1}$. Then we have the following
	identifications:
  \begin{itemize}
		\item We send $-\infty$ to the end $(w_m(-i))_{i \in \N}$, i.e.\  we
			walk the bi-infinite line in the negative direction.
		\item We send $+\infty$ to the end $(w_m(+i))_{i \in \N}$, i.e.\  we
			walk the bi-infinite line in the positive direction.
    \item Given a pair $(z, v) \in \Z \times \partial
			\overline{\Gamma}_{m}$, we identify it with the end which is given
			by the concatenation of the path from $w_m0$ to $w_mz$ together
			with the sequence $v_nm$. This means that first walk to the root
			of the copy of $\overline{\Gamma}_{m}$ labeled by $z$, and then
			go the end defined by the sequence $v$ in this copy.
\end{itemize}
Using Lemma~\ref{lem:induct}, it is easy to check that this indeed
defines a homeomorphism as given in \eqref{eq:endinduct}.
Now is a $\Z \times \partial
		\overline{\Gamma}_{m} \cup \left\{ -\infty, +\infty \right\}$ is
countable union of countable set, so $\partial
		\overline{\Gamma}_{m+1}$ is countable.
\end{proof}
Let us fix some notation related to the orbital Schreier graph
$\overline{\Gamma}_\omega$. For $u \in \Z^\omega$, let $T_m(u)$ be the induced
subgraph of $\overline{\Gamma}_\omega$ on the set $\left\{ u' \in \Z^\omega
\colon u_i = u'_i\text{ for all }i > m  \right\}$. We denote the union $\bigcup_{m\in
\N} T_m(u)$ by $T(u)$.
\begin{theorem}
  The connected component of $u$ in $\overline{\Gamma}_\omega$ is $T(u)$. It is
  a tree with countably many ends.
  \label{thm:schreier}
\end{theorem}

\begin{proof}
  The projection to the prefix of length $m$ is a bijection from the vertex set
  of $T_m(u)$ to $\Z^m$. It gives rise to graph isomorphism from $T_m(u)$ to
  $\overline{\Gamma}_m$, as the generating set acts by changing at most one
  letter at once. So $T(u)$ is an increasing union of trees, hence it is
  also a tree.

  Each end of $T(u)$ either stays in some $T_m(u)$ or leaves all $T_m(u)$. The
  first kind is a countable union of countable sets, hence we only need to
  consider ends leaving all $T_m(u)$. Let $E_m$ be the set of edges in $T(u)$
  leaving $T_m(u)$. We have a map $E_m \rightarrow E_{m-1}$ which sends an edge $e$
  leaving $T_m(u)$ to the unique edge leaving $T_{m-1}(u)$ on the geodesic from
  $u$ to $e$. It is possible that an edge is send to itself, if it leaves
  multiple subtrees at once. Now the set of ends leaving all $T_m(u)$ is
  isomorphic to $\varprojlim E_m$.  
  Now the sets $E_m$ have uniform bounded cardinality. This can be seen as
  follows:
  Let $w$ be the $m$-suffix of $u$. Then an edge in $E_m$ corresponds to a pair
  $v \in \Z^m, q \in S \cup S^{-1}$ with $q|v(w) \not= w$, in particular the
  restriction $q|v$ is not trivial. But $\Kaut(x_1\dots x_k, y_1 \dots y_p)$ is
  a bounded activity automaton, so the number of pairs $(v,q) \in \Z^m \times
  \left( S \cup S^{-1} \right)$ with $q|v \not= \Id$ is uniformly bounded, and
  so are the sets $E_m$.
  Hence the inverse limit has finite cardinality, so in total we have countably
  many ends.
  
\end{proof}

\section{Group theoretic properties}
The groups $\Kgroup(x_1\dots x_k, y_1 \dots y_p)$ are examples of ZC-groups
defined as \cite{oliynyk2010groups}. In particular, they are left-orderable residually solvable groups.

In this section, we will always work with a fixed pair of sequences $x_1\dots x_k, y_1 \dots y_p$
and we will just write $\Kgroup$ instead of $\Kgroup(x_1\dots x_k, y_1 \dots y_p)$. We still use
$S \coloneqq \left\{ a_1, \dots, a_k, b_1, \dots, b_p  \right\}$ as our generating set.

\begin{lemma}
  The abelianization of $\Kgroup(x_1\dots x_k, y_1 \dots y_p)$ is the free
  abelian group on $x_1,\dots x_k, y_1 \dots y_p$.
\end{lemma}
\begin{proof}
  We have a family of group homomorphisms
  \begin{eqnarray*}
    \overline \rho_n \colon \aut_\Z(\Z^*) & \rightarrow & \Z  \\
    g & \mapsto & \sum_{v \in \Z^n} \rho(g_{|v})
  \end{eqnarray*}
  Note that the sum is defined as $g_{|v}$ is trivial for almost all $v$, so almost all summands
  are $0$. By the cocycle equations~\ref{eqn:cocycle} we see that $\overline \rho_n$ is indeed a group homomorphism,
  and for all $g \in \aut_\Z(\Z^*)$, we have $\overline \rho_{n + 1} (g) = \sum_{v \in \Z} \overline \rho_n (g_{|v})$.
    The transition functions given in the Definition~\ref{def:kneading} translate to
  \begin{eqnarray*}
    \overline \rho_0(a_1) &=& 1 \\
    \overline \rho_0(s) &=& 0 \text{~for all~} s \in S \setminus (a_0)\\
    \overline \rho_{n+1}(a_{i+1}) &=& \overline \rho_{n} (a_i) \\
    \overline \rho_{n+1}(b_{1}) &=& \overline \rho_{n} (a_k) + \overline \rho_{n} (b_p) \\
    \overline \rho_{n+1}(b_{j+1}) &=& \overline \rho_{n} (a_j) \\
  \end{eqnarray*}
    If we collect $\overline \rho_{0}, \dots \overline \rho_{k+p-1}$ to a group
    homomorphism $\overline \rho \colon \aut_\Z(\Z^*) \rightarrow \Z^{k+p}$, we
    can show row by row that $\left( \overline \rho(a_1), \dots , \overline
    \rho(a_k), \overline \rho (b_1), \dots , \overline \rho(b^k)\right) \in
    \Z^{(k+p)\times(k+p)}$ is the identity matrix. So $\overline \rho$ induces
    an isomorphism between the abelianization of $\Kgroup$ and $\Z^{ k+p }$.
\end{proof}
\begin{lemma}
  $\Kgroup$ surjects onto the restricted wreath product $\Z \wr \Z$. In particular, $\Kgroup$ is of exponential growth.
  \label{lem:growth}
\end{lemma}
\begin{proof}
  The action on $\Z^2$ gives a map to $\Z \wr \Z$. We see that $a_1$ and $a_2$ respectively $b_1$ if $k = 1$ are mapped to the standard generating
  set of $\Z \wr \Z$, so we have a surjection. As $\Z \wr \Z$ has
  exponential growth (see \cite{parry1992,Bucher2017} for a detailed discussion),	$\Kgroup$ also has
  exponential growth. 
\end{proof}
\begin{lemma}
  The group $\Kgroup(x_1\dots x_k, y_1 \dots y_p)$ is level-transitive
	and self-replicating. For the derived subgroup $\Kgroup' \subset
	\Stab_\Kgroup$ we have the following: under the map $\Stab_\Kgroup
	\hookrightarrow \bigoplus_{x \in \Z} \Kgroup$ induced by the wreath
	recursion, the image of $\Kgroup'$ contains $\bigoplus_{x \in \Z}
\Kgroup'$ and the composition
\begin{equation}
	\Kgroup' \hookrightarrow \Stab_\Kgroup \hookrightarrow \bigoplus_{x
	\in \Z} \Kgroup \rightarrow \Kgroup
	\label{eq:surj}
\end{equation}
is surjective, where the last map is the projection map to any summand.
  \label{lem:branch}
\end{lemma}
\begin{proof}
  Note that $a \coloneqq a_1$ acts just by translations on the first level, and every
  generator is the section of another generator. This already implies
  level-transitive and self-replicating.
	To show that the composition \eqref{eq:surj} is surjective, it is easy
	to see that every generator of $\Kgroup$ is a section of a commutator
	of a generator and a sufficiently large power of
  $a_1$. So it is easy to see that $\Kgroup'$ surjects geometrically onto
  $\Kgroup$. As $a_1$ is just the first level shift, and $\Kgroup'$ is a normal
  subgroup of $\Kgroup$, to show that $\bigoplus_{x \in \Z} \Kgroup' \subset
  \Kgroup'$, it is enough to show that $\Kgroup' @ 0 \subset \Kgroup'$. Since
  $\Kgroup$ is self-replicating, it is enough to show that $[s,t]@0 \in
  \Kgroup'$ for every commutator of two generators $s,t \in S$
	Now if $c$ and $d$ are the generators
  which have $s$ and $t$ as sections at $z$ and $w$, then a straight forward calculation
  shows $[a^{-z}ca^{z},a^{-w}da^{w}] = [s,t] @ 0$.
\end{proof}
\begin{lemma}
  The groups $\Kgroup(x_1\dots x_k, y_1 \dots y_p)$ are not residually finite.
  \label{lem:resfinite}
\end{lemma}
\begin{proof}
  By the previous lemma, $\Stab_\Kgroup$ surjects onto $\Kgroup$, and since $\Kgroup$
  is not abelian (it surjects onto an non-abelian group), neither is $\Stab_\Kgroup$.
  Let $x,y \in \Stab_\Kgroup$ be a non-commuting pair. Suppose $\Kgroup$ is residually finite,
  then there exists a group homomorphism $\phi \colon \Kgroup \rightarrow F$ to a finite group $F$ such that
  $\phi([x,y])$ is non-trivial. But $F$ is finite, so $\phi(a_1)$ has finite order. So there is a $n > 0$
  With $\phi(a^{mn}_1) = 1$ for all $m$. Then $\phi([x,y]) = \phi([a^{-mn}_1xa^{mn}_1,y])$.
  Now under the wreath recursion, $x$ and $y$ have finite support in the direct sum
  $\bigoplus_{\Z} \aut_\Z(\Z)$, so for $m$ large enough, the support of $a^{-mn}_1xa^{mn}_1$ and $y$
  will be disjoint, hence they commute. So $\phi([x,y]) = \phi([a^{-mn}_1xa^{mn}_1,y])$ is trivial, so we arrive at an contradiction.
\end{proof}
\begin{theorem}
  The groups $\Kgroup(x_1 \dots x_k, y_1 \dots y_p)$ are amenable but
	not elementary subexponentially amenable. 
  \label{thm:amenable}
\end{theorem}
\begin{proof}
  We invoke Theorem~B of \cite{Reinkegroup} to show that the groups $\Kgroup(x_1 \dots x_k, y_1 \dots y_p)$
  are amenable. We already observed in Remark~\ref{rem:boundedactivity} that the groups are generated by bounded activity automata.
  Hence they are subgroups of $\autBfs(\Z^*;\Z)$. As the left action of $\Z$ on itself is recurrent, by Theorem~B of \cite{Reinkegroup} the group $\autBfs(\Z^*;\Z)$ is recurrent, and so are the subgroups $\Kgroup(x_1 \dots x_k, y_1 \dots y_p)$.

  The groups have exponential growth by Lemma~\ref{lem:growth}. Lemma~\ref{lem:branch}
	together with Corollary~3 of \cite{juschenko2018} imply that the
	groups are not elementary subexponentially amenable. 
\end{proof}
We should note that \cite{juschenko2018} only deals with finite
alphabets. The proof can be easily modified to deal with subgroups of
$\aut_\Z(\Z^*)$.

\newpage

\section{Outlook}
This paper is the beginning of our study of iterated monodromy groups for entire transcendental maps and
a stepping stone towards a more general discussion. The regularity of the monodromy of the exponential map simplifies
the discussion and has consequences that are special to the exponential case. In particular, the left-order on
the dynamical preimage tree heavily uses this regularity. For other entire transcendental functions, we should expect
torsion elements in the monodromy group and torsion elements for some iterated monodromy groups of functions in that parameter space.

In an upcoming paper \cite{Reinkeentire} we discuss the general structure of iterated monodromy groups of entire maps. In particular, we also apply
the results of \cite{Reinkegroup} to show that the iterated monodromy groups of entire functions are amenable if and only their monodromy group is.
For polynomials and the exponential family, the condition is trivially satisfied, as finite groups and abelian groups are amenable. However, there are entire maps with virtually free monodromy groups, so we have to impose this condition.

Moreover, we can also try to generalize from entire functions to meromorphic functions. Here a good starting family would be the functions of
the form $M \circ \exp$ including tangent, where $M$ is a Möbius transform. We should
think of this as the analogy to the family of bicritical rational maps,
see also Appendix D of~\cite{milnor2000}. In this case, we can also define iterated monodromy group for post-singularly finite maps and show that they are ZC-groups. So the class of ZC-groups, in particular subgroups of $\aut_\Z(\Z^*)$ has many examples of self-similar groups coming from complex dynamics. This warrants a further general investigation of ZC-groups.

Outside of this family $M \circ \exp$, we should not expect to have
the left-orderability of all IMGs in one parameter space, as it might be a special
phenomenon due to the very rigid monodromy groups of exponential maps.
\bibliographystyle{alpha}
\bibliography{../bibfile.bib}

\newcommand{\etalchar}[1]{$^{#1}$}
\begin{thebibliography}{BDH{\etalchar{+}}00}

\bibitem[BDH{\etalchar{+}}00]{DGH}
Clara Bodel\'{o}n, Robert~L. Devaney, Michael Hayes, Gareth Roberts, Lisa~R.
  Goldberg, and John~H. Hubbard.
\newblock Dynamical convergence of polynomials to the exponential.
\newblock {\em J. Differ. Equations Appl.}, 6(3):275--307, 2000.

\bibitem[BN06]{Bartholdi2006}
Laurent Bartholdi and Volodymyr Nekrashevych.
\newblock Thurston equivalence of topological polynomials.
\newblock {\em Acta Math.}, 197(1):1--51, 2006.

\bibitem[BN08]{Bartholdi2008kneading}
Laurent Bartholdi and Volodymyr Nekrashevych.
\newblock Iterated monodromy groups of quadratic polynomials. {I}.
\newblock {\em Groups Geom. Dyn.}, 2(3):309--336, 2008.

\bibitem[BSc02]{bruinsymbolic}
Henk Bruin and Dierk Schlei\-cher.
\newblock Symbolic dynamics of quadratic polynomials.
\newblock Technical Report~7, Institut Mittag-Leffler, 2001/2002.

\bibitem[BT17]{Bucher2017}
Michelle Bucher and Alexey Talambutsa.
\newblock Minimal exponential growth rates of metabelian {B}aumslag-{S}olitar
  groups and lamplighter groups.
\newblock {\em Groups Geom. Dyn.}, 11(1):189--209, 2017.

\bibitem[FG91]{FabrykowskiGupta}
Jacek Fabrykowski and Narain Gupta.
\newblock On groups with sub-exponential growth functions. {II}.
\newblock {\em J. Indian Math. Soc. (N.S.)}, 56(1-4):217--228, 1991.

\bibitem[G{\.{Z}}02]{GZ2002}
Rostislav~I. Grigorchuk and Andrzej {\.{Z}}uk.
\newblock On a torsion-free weakly branch group defined by a three state
  automaton.
\newblock volume~12, pages 223--246. 2002.
\newblock International Conference on Geometric and Combinatorial Methods in
  Group Theory and Semigroup Theory (Lincoln, NE, 2000).

\bibitem[HSS09]{HSS}
John Hubbard, Dierk Schleicher, and Mitsuhiro Shishikura.
\newblock Exponential {T}hurston maps and limits of quadratic differentials.
\newblock {\em J. Amer. Math. Soc.}, 22(1):77--117, 2009.

\bibitem[Jus18]{juschenko2018}
Kate Juschenko.
\newblock Non-elementary amenable subgroups of automata groups.
\newblock {\em J. Topol. Anal.}, 10(1):35--45, 2018.

\bibitem[LSV08]{laubner2008}
Bastian Laubner, Dierk Schleicher, and Vlad Vicol.
\newblock A combinatorial classification of postsingularly finite complex
  exponential maps.
\newblock {\em Discrete Contin. Dyn. Syst.}, 22(3):663--682, 2008.

\bibitem[Mil00]{milnor2000}
John Milnor.
\newblock On rational maps with two critical points.
\newblock {\em Experiment. Math.}, 9(4):481--522, 2000.

\bibitem[MT88]{milnor1988iterated}
John Milnor and William Thurston.
\newblock On iterated maps of the interval.
\newblock In {\em Dynamical systems}, pages 465--563. Springer, 1988.

\bibitem[Nek05]{nekrashevych2005self}
Volodymyr Nekrashevych.
\newblock {\em Self-similar groups}, volume 117 of {\em Mathematical Surveys
  and Monographs}.
\newblock American Mathematical Society, Providence, RI, 2005.

\bibitem[OS10]{oliynyk2010groups}
AS~Oliynyk and VI~Sushchanski{\u\i}.
\newblock The groups of {ZC}-automaton transformations.
\newblock {\em Siberian mathematical journal}, 51(5):879--891, 2010.

\bibitem[Par92]{parry1992}
Walter Parry.
\newblock Growth series of some wreath products.
\newblock {\em Trans. Amer. Math. Soc.}, 331(2):751--759, 1992.

\bibitem[Rei]{Reinkeentire}
Bernhard Reinke.
\newblock Iterated monodromy groups of entire functions.
\newblock In preperation.

\bibitem[Rei20]{Reinkegroup}
Bernhard Reinke.
\newblock Amenability of bounded automata groups on infinite alphabets.
\newblock arXiv:2004.05029.

\bibitem[Sch10]{schleicher2010dynamics}
Dierk Schleicher.
\newblock Dynamics of entire functions.
\newblock In Graziano Gentili, Jacques Guenot, and Giorgio Patrizio, editors,
  {\em Holomorphic dynamical systems}, volume 1998 of {\em Lecture Notes in
  Math.}, pages 295--339. Springer, Berlin, 2010.

\bibitem[Sid04]{sidki2004finite}
Said Sidki.
\newblock Finite automata of polynomial growth do not generate a free group.
\newblock {\em Geometriae Dedicata}, 108(1):193--204, 2004.

\bibitem[SZ03]{Zimmer2003}
Dierk Schleicher and Johannes Zimmer.
\newblock Periodic points and dynamic rays of exponential maps.
\newblock {\em Ann. Acad. Sci. Fenn. Math.}, 28(2):327--354, 2003.

\bibitem[Thu09]{thurston2009}
William~P. Thurston.
\newblock On the geometry and dynamics of iterated rational maps.
\newblock In {\em Complex dynamics}, pages 3--137. A K Peters, Wellesley, MA,
  2009.
\newblock Edited by Dierk Schleicher and Nikita Selinger and with an appendix
  by Schleicher.

\end{thebibliography}
\end{document}